\documentclass[12pt,reqno]{amsart} 
\usepackage{amssymb,amscd,url}

\begin{document}

\allowdisplaybreaks

\newif\ifdraft 
\drafttrue
\draftfalse
\newcommand{\DRAFTNUMBER}{1}
\newcommand{\DATE}{\today\ \ifdraft(Draft \DRAFTNUMBER)\fi}
\newcommand{\TITLE}{Lang's Height Conjecture and Szpiro's Conjecture}
\newcommand{\TITLERUNNING}{Lang's Height Conjecture and Szpiro's Conjecture}



\newtheorem{theorem}{Theorem}
\newtheorem{lemma}[theorem]{Lemma}
\newtheorem{conjecture}[theorem]{Conjecture}
\newtheorem{proposition}[theorem]{Proposition}
\newtheorem{corollary}[theorem]{Corollary}
\newtheorem*{claim}{Claim}

\theoremstyle{definition}
\newtheorem{question}{Question}
\renewcommand{\thequestion}{\Alph{question}} 
\newtheorem*{definition}{Definition}
\newtheorem{example}[theorem]{Example}

\theoremstyle{remark}
\newtheorem{remark}[theorem]{Remark}
\newtheorem*{acknowledgement}{Acknowledgements}



\newenvironment{notation}[0]{%
  \begin{list}%
    {}%
    {\setlength{\itemindent}{0pt}
     \setlength{\labelwidth}{4\parindent}
     \setlength{\labelsep}{\parindent}
     \setlength{\leftmargin}{5\parindent}
     \setlength{\itemsep}{0pt}
     }%
   }%
  {\end{list}}

\newenvironment{parts}[0]{%
  \begin{list}{}%
    {\setlength{\itemindent}{0pt}
     \setlength{\labelwidth}{1.5\parindent}
     \setlength{\labelsep}{.5\parindent}
     \setlength{\leftmargin}{2\parindent}
     \setlength{\itemsep}{0pt}
     }%
   }%
  {\end{list}}
\newcommand{\Part}[1]{\item[\upshape#1]}
\newcommand{\ProofPart}[1]{\par\noindent\textup{#1}\enspace\ignorespaces}

\renewcommand{\a}{\alpha}
\renewcommand{\b}{\beta}
\newcommand{\g}{\gamma}
\renewcommand{\d}{\delta}
\newcommand{\e}{\epsilon}
\newcommand{\f}{\varphi}
\newcommand{\bfphi}{{\boldsymbol{\f}}}
\renewcommand{\l}{\lambda}
\renewcommand{\k}{\kappa}
\newcommand{\lhat}{\hat\lambda}
\newcommand{\m}{\mu}
\newcommand{\bfmu}{{\boldsymbol{\mu}}}
\renewcommand{\o}{\omega}
\renewcommand{\r}{\rho}
\newcommand{\rbar}{{\bar\rho}}
\newcommand{\s}{\sigma}
\newcommand{\sbar}{{\bar\sigma}}
\renewcommand{\t}{\tau}
\newcommand{\z}{\zeta}

\newcommand{\D}{\Delta}
\newcommand{\G}{\Gamma}
\newcommand{\F}{\Phi}

\newcommand{\ga}{{\mathfrak{a}}}
\newcommand{\gA}{{\mathfrak{A}}}
\newcommand{\gb}{{\mathfrak{b}}}
\newcommand{\gB}{{\mathfrak{B}}}
\newcommand{\gc}{{\mathfrak{c}}}
\newcommand{\gC}{{\mathfrak{C}}}
\newcommand{\gd}{{\mathfrak{d}}}
\newcommand{\gD}{{\mathfrak{D}}}
\newcommand{\gI}{{\mathfrak{I}}}
\newcommand{\gm}{{\mathfrak{m}}}
\newcommand{\gn}{{\mathfrak{n}}}
\newcommand{\go}{{\mathfrak{o}}}
\newcommand{\gO}{{\mathfrak{O}}}
\newcommand{\gp}{{\mathfrak{p}}}
\newcommand{\gP}{{\mathfrak{P}}}
\newcommand{\gq}{{\mathfrak{q}}}
\newcommand{\gR}{{\mathfrak{R}}}

\newcommand{\Abar}{{\bar A}}
\newcommand{\Ebar}{{\bar E}}
\newcommand{\Kbar}{{\bar K}}
\newcommand{\Pbar}{{\bar P}}
\newcommand{\Sbar}{{\bar S}}
\newcommand{\Tbar}{{\bar T}}
\newcommand{\ybar}{{\bar y}}
\newcommand{\phibar}{{\bar\f}}

\newcommand{\Acal}{{\mathcal A}}
\newcommand{\Bcal}{{\mathcal B}}
\newcommand{\Ccal}{{\mathcal C}}
\newcommand{\Dcal}{{\mathcal D}}
\newcommand{\Ecal}{{\mathcal E}}
\newcommand{\Fcal}{{\mathcal F}}
\newcommand{\Gcal}{{\mathcal G}}
\newcommand{\Hcal}{{\mathcal H}}
\newcommand{\Ical}{{\mathcal I}}
\newcommand{\Jcal}{{\mathcal J}}
\newcommand{\Kcal}{{\mathcal K}}
\newcommand{\Lcal}{{\mathcal L}}
\newcommand{\Mcal}{{\mathcal M}}
\newcommand{\Ncal}{{\mathcal N}}
\newcommand{\Ocal}{{\mathcal O}}
\newcommand{\Pcal}{{\mathcal P}}
\newcommand{\Qcal}{{\mathcal Q}}
\newcommand{\Rcal}{{\mathcal R}}
\newcommand{\Scal}{{\mathcal S}}
\newcommand{\Tcal}{{\mathcal T}}
\newcommand{\Ucal}{{\mathcal U}}
\newcommand{\Vcal}{{\mathcal V}}
\newcommand{\Wcal}{{\mathcal W}}
\newcommand{\Xcal}{{\mathcal X}}
\newcommand{\Ycal}{{\mathcal Y}}
\newcommand{\Zcal}{{\mathcal Z}}

\renewcommand{\AA}{\mathbb{A}}
\newcommand{\BB}{\mathbb{B}}
\newcommand{\CC}{\mathbb{C}}
\newcommand{\FF}{\mathbb{F}}
\newcommand{\GG}{\mathbb{G}}
\newcommand{\NN}{\mathbb{N}}
\newcommand{\PP}{\mathbb{P}}
\newcommand{\QQ}{\mathbb{Q}}
\newcommand{\RR}{\mathbb{R}}
\newcommand{\ZZ}{\mathbb{Z}}

\newcommand{\bfa}{{\mathbf a}}
\newcommand{\bfb}{{\mathbf b}}
\newcommand{\bfc}{{\mathbf c}}
\newcommand{\bfe}{{\mathbf e}}
\newcommand{\bff}{{\mathbf f}}
\newcommand{\bfg}{{\mathbf g}}
\newcommand{\bfp}{{\mathbf p}}
\newcommand{\bfr}{{\mathbf r}}
\newcommand{\bfs}{{\mathbf s}}
\newcommand{\bft}{{\mathbf t}}
\newcommand{\bfu}{{\mathbf u}}
\newcommand{\bfv}{{\mathbf v}}
\newcommand{\bfw}{{\mathbf w}}
\newcommand{\bfx}{{\mathbf x}}
\newcommand{\bfy}{{\mathbf y}}
\newcommand{\bfz}{{\mathbf z}}
\newcommand{\bfA}{{\mathbf A}}
\newcommand{\bfF}{{\mathbf F}}
\newcommand{\bfB}{{\mathbf B}}
\newcommand{\bfD}{{\mathbf D}}
\newcommand{\bfG}{{\mathbf G}}
\newcommand{\bfI}{{\mathbf I}}
\newcommand{\bfM}{{\mathbf M}}
\newcommand{\bfzero}{{\boldsymbol{0}}}

\newcommand{\Adele}{\textsf{\upshape A}}
\newcommand{\Aut}{\operatorname{Aut}}
\newcommand{\Bern}{\mathbb{B}} 
\newcommand{\Br}{\operatorname{Br}}  
\newcommand{\Cond}{\mathfrak{F}}
\newcommand{\Disc}{\mathfrak{D}}
\newcommand{\density}{{\boldsymbol\delta}}
\newcommand{\densitysup}{\overline{\density}}
\newcommand{\densityinf}{\underline{\density}}
\newcommand{\ds}{\displaystyle}
\newcommand{\dsum}{\displaystyle\sum}
\newcommand{\Div}{\operatorname{Div}}
\renewcommand{\div}{{\operatorname{div}}}
\newcommand{\End}{\operatorname{End}}
\newcommand{\Fbar}{{\bar{F}}}
\newcommand{\FOD}{\textup{FOM}}
\newcommand{\FOM}{\textup{FOD}}
\newcommand{\Gal}{\operatorname{Gal}}
\newcommand{\GL}{\operatorname{GL}}
\newcommand{\Index}{\operatorname{Index}}
\newcommand{\Image}{\operatorname{Image}}
\newcommand{\liftable}{{\textup{liftable}}}
\newcommand{\hhat}{{\hat h}}
\newcommand{\Ksep}{K^{\textup{sep}}}
\newcommand{\Ker}{{\operatorname{ker}}}
\newcommand{\Lsep}{L^{\textup{sep}}}
\newcommand{\Lift}{\operatorname{Lift}}
\newcommand{\pp}{\operatorname{pp}}  
\newcommand{\vlim}{\operatornamewithlimits{\text{$v$}-lim}}
\newcommand{\wlim}{\operatornamewithlimits{\text{$w$}-lim}}
\newcommand{\MOD}[1]{~(\textup{mod}~#1)}
\newcommand{\Norm}{{\operatorname{\mathsf{N}}}}
\newcommand{\notdivide}{\nmid}
\newcommand{\normalsubgroup}{\triangleleft}
\newcommand{\odd}{{\operatorname{odd}}}
\newcommand{\onto}{\twoheadrightarrow}
\newcommand{\Orbit}{\mathcal{O}}
\newcommand{\ord}{\operatorname{ord}}
\newcommand{\Per}{\operatorname{Per}}
\newcommand{\PrePer}{\operatorname{PrePer}}
\newcommand{\PGL}{\operatorname{PGL}}
\newcommand{\Pic}{\operatorname{Pic}}
\newcommand{\Prob}{\operatorname{Prob}}
\newcommand{\Qbar}{{\bar{\QQ}}}
\newcommand{\rank}{\operatorname{rank}}
\newcommand{\Resultant}{\operatorname{Res}}
\renewcommand{\setminus}{\smallsetminus}
\newcommand{\SL}{\operatorname{SL}}
\newcommand{\Span}{\operatorname{Span}}
\newcommand{\Spec}{\operatorname{Spec}}
\newcommand{\tors}{{\textup{tors}}}
\newcommand{\Trace}{\operatorname{Trace}}
\newcommand{\twistedtimes}{\mathbin{%
   \mbox{$\vrule height 6pt depth0pt width.5pt\hspace{-2.2pt}\times$}}}
\newcommand{\UHP}{{\mathfrak{h}}}    
\newcommand{\Wreath}{\operatorname{Wreath}}
\newcommand{\<}{\langle}
\renewcommand{\>}{\rangle}

\newcommand{\longhookrightarrow}{\lhook\joinrel\longrightarrow}
\newcommand{\longonto}{\relbar\joinrel\twoheadrightarrow}

\newcounter{CaseCount}
\Alph{CaseCount}
\def\Case#1{\par\vspace{1\jot}\noindent
\stepcounter{CaseCount}
\framebox{Case \Alph{CaseCount}.\enspace#1}
\par\vspace{1\jot}\noindent\ignorespaces}

\title[\TITLERUNNING]{\TITLE}
\date{\DATE}

\author[Joseph H. Silverman]{Joseph H. Silverman}
\email{jhs@math.brown.edu}
\address{Mathematics Department, Box 1917
         Brown University, Providence, RI 02912 USA}
\thanks{The author's research supported by NSF grant DMS-0650017 and
DMS-0854755.}

\begin{abstract}
It is known that Szpiro's conjecture, or equivalently the
$ABC$-conjecture, implies Lang's conjecture giving a uniform lower
bound for the canonical height of nontorsion points on elliptic
curves. In this note we show that a significantly weaker version of
Szpiro's conjecture, which we call ``prime-depleted,'' suffices to
prove Lang's conjecture.
\end{abstract}


\maketitle

\section*{Introduction}

Let $E/K$ be an elliptic curve defined over a number field, let $P\in
E(K)$ be a nontorsion point on~$E$, and write~$\Disc(E/K)$
and~$\Cond(E/K)$ for the discriminant and the conductor of~$E/K$. In
this paper we discuss the relationship between the following
conjectures of Serge Lang~\cite[page~92]{LangECDA} and Lucien
Szpiro~(1983).

\begin{conjecture}[Lang Height Conjecture]
There are constants $C_1=C_1(K)>0$ and~$C_2=C_2(K)$ such that the canonical
height of~$P$ is bounded below by
\[
  \hhat(P) \ge C_1\log\Norm_{K/\QQ}\Disc(E/K) - C_2.
\]
\end{conjecture}

\begin{conjecture}[Szpiro Conjecture]
There are constants~$C_3$ and~$C_4=C_4(K)$ such that
\[
  \log\Norm_{K/\QQ}\Disc(E/K)
  \le C_3\log\Norm_{K/\QQ}\Cond(E/K) + C_4.
\]
\end{conjecture}

In~\cite{MR948108} Marc Hindry and the author proved that Szpiro's
conjecture implies Lang's height
conjecture. (See~\cite{MR1450488,MR2259240} for improved constants.)
It is thus tempting to try to prove the opposite implication, i.e.,
prove that Lang's height conjecture implies Szpiro's conjecture.
Since Szpiro's conjecture is easily seen to be imply the
$ABC$-conjecture of Masser and Oesterl\'e~\cite{Oesterle} (with some
exponent), such a proof would be of interest.  \par It is the purpose
of this note to explain how the pigeonhole argument in~\cite{MR630588}
may be combined with the Fourier averaging methods in~\cite{MR948108}
to prove Lang's height conjecture using a weaker from of Szpiro's
conjecture. Roughly speaking, the ``prime-depleted'' version of
Szpiro's conjecture that we require allows one to discard a bounded
number of primes from~$\Disc(E/K)$ and~$\Cond(E/K)$ before comparing
them. 
\par
We briefly summarize the contents of this paper.  In
Section~\ref{section:pdszpiro} we describe the prime-depleted Szpiro
conjecture and prove that it implies Lang's height conjecture.
Section~\ref{section:properties} contains various elementary
properties of the prime-depleted Szpiro ratio. Finally, in
Section~\ref{section:abc} we state a prime-depleted $ABC$-conjecture
and show that it is a consquence of the prime-depleted Szpiro
conjecture.

\section{The Prime-Depleted Szpiro Conjecture}
\label{section:pdszpiro}

We begin with some definitions. 

\begin{definition}
Let~$\gD$ be an integral ideal and factor
$\gD=\prod\gp^{e_i}$ as a product of prime powers. We write $\nu(\gD)$
for the number of factors in the product, i.e.,~$\nu(\gD)$ is the
number of distinct prime ideals dividing~$\gD$.  The \emph{Szpiro
ratio of~$\gD$} is the quantity
\[
  \s(\gD) = \frac{\log\Norm_{K/\QQ}\gD}{\log\Norm_{K/\QQ}\prod_i\gp_i}
  = \frac{\sum e_i\log\Norm_{K/\QQ}\gp_i}{\sum \log\Norm_{K/\QQ}\gp_i}.
\]
(If $\gD=(1)$, we set $\s(\gD)=1$.)
More generally, for any integer~$J\ge0$, the \emph{$J$-depleted
Szpiro ratio of~$\gD$} is defined as follows:
\[
  \s_J(\gD) = \min_{\substack{I\subset\{1,2,\ldots,\nu(\gD)\}\\
        \#I\ge\nu(\gD)-J\\}} 
        \s\biggl(\prod_{i\in I} \gp_i^{e_i}\biggr).
\]
Thus~$\s_J(\gD)$ is the smallest value that can be obtained by
removing up to~$J$ of the prime powers divided~$\gD$ before computing
the Szpiro ratio.  We note that if $\nu(\gD)\le J$, then $\s_J(\gD)=1$
by definition.
\end{definition}

\begin{example}
\[
  \s_0(1728)=\frac{\log1728}{\log6}\approx4.16,
  \quad
  \s_1(1728)=\frac{\log27}{\log3} = 3,
  \quad
  \s_2(1728)=1.
\]
\end{example}

\begin{conjecture}[Prime-Depleted Szpiro Conjecture]
Let $K/\QQ$ be a number field.  There exist an integer $J\ge0$ and a
constant $C_3$, depending only on~$K$, such that for all elliptic
curves~$E/K$,
\[
  \s_J\bigl(\Disc(E/K)\bigr) \le C_3.
\]
\end{conjecture}

It is clear from the definition that $\s_0(\gD)=\s(\gD)$.  An
elementary argument (Proposition~\ref{proposition:sproperties}) shows
that the value of~$\s_J$ decreases as~$J$ increases,
\[
  \s_0(\gD)\ge\s_1(\gD)\ge\s_2(\gD)\ge\cdots\,.
\]
Hence the prime-depleted Szpiro conjecture is weaker than the
classical version, which says that~$\s_0\bigl(\Disc(E/K)\bigr)$ is
bounded independent of~$E$.  Before stating our main result, we need
one further definition.

\begin{definition}
Let $E/K$ be an elliptic curve defined over a number field.
The \emph{height of $E/K$} is the quantity
\[
  h(E/K) = \max\bigl\{h\bigl(j(E)\bigr),
     \log\Norm_{K/\QQ}\Disc(E/K)\bigr\}.
\]
For a given field~$K$, there are only finitely many elliptic
curves~$E/K$ of bounded height, although there may be infinitely many
elliptic curves of bounded height defined over fields of bounded
degree~\cite{MR953747}.
\end{definition}

We now state our main result, which implies that the $J$-depleted
Szpiro conjecture implies Lang's height conjecture.

\begin{theorem}
Let $K/\QQ$ be a number field, let $J\ge1$ be an integer, let~$E/K$ be
an elliptic curve, and let $P\in E(K)$ be a nontorsion point.  There
are constants $C_1>0$ and $C_2$, depending only on~$[K:\QQ]$,~$J$, and
the $J$-depleted Szpiro ratio~$\s_J\bigl(\Disc(E/K)\bigr)$, such that
\[
  \hhat(P) \ge C_1h(E/K) - C_2.
\]
In particular, the depleted Szpiro conjecture implies Lang's height
conjecture.
\end{theorem}

\begin{remark}
As in~\cite{MR2259240}, it is not hard to give explicit expressions
for~$C_1$ and~$C_2$ in terms of~$[K:\QQ]$,~$J$,
and~$\s_J\bigl(\Disc(E/K)\bigr)$, but we will not do so here. In terms
of the dependence on the Szpiro ratio, probably the best that comes
out of a careful working of the proof is something like
\[
  C_1 \gg\ll \s_J\bigl(\Disc(E/K)\bigr)^{cJ}
\]
for some absolute constant~$c$.  But until the (depleted) Szpiro
conjecture is proven or a specific application arises, such explict
expressions seem of limited utility.
\end{remark}

\begin{proof}
We refer the reader to~\cite[Chapter~6]{MR1312368} for basic material
on canonical local heights on elliptic curves.
Replacing~$P$ with~$12P$, we may assume without loss of generality
that the local height satisfies 
\[
  \lhat(P;v)\ge\frac{1}{12}\log\Norm_{K/\QQ}\Disc(E/K)
\]
for all nonarchimedean places~$v$ at which~$E$ does not have split
multiplicative reduction.  We factor the discriminant $\Dcal(E/K)$
into a product
\[
  \gD(E/K) = \gD_1\gD_2
  \quad\text{with}\quad
  \nu(\gD_2)\le J
  \quad\text{and}\quad
  \s_J\bigl(\gD(E/K)\bigr)=\s(\gD_1).
\]
We also choose an integer~$M\ge1$ whose value will be specified later,
and for convenience we let $d=[K:\QQ]$.
\par
Using a pigeon-hole principle argument as described
in~\cite{MR630588}, we can find an integer~$k$ with
\[
  1 \le k \le (6M)^{J+d}
\]
such that for all $1\le m\le M$ we have
\begin{align*}
  \lhat(mkP;v) &\ge c_1 \log\max\bigl\{ |j(E)|_v , 1 \bigr\}  -  c_2
  \quad\text{for all $v\in M_K^\infty$,} \\
  \lhat(mkP;v) &\ge c_3 \log\bigl|\Norm_{K/\QQ}\gD(E/K)\bigr|_v^{-1}
  \quad\text{for all $v\in M_K^0$ with $\gp_v\mid\gD_2$.}
\end{align*}
(Here and in what follows, $c_1,c_2,\dots$ are absolute positive
constants.) Roughly speaking, we need to force $J+d$ local heights
to be positive for all~$mP$ with $1\le m\le M$, which is why we
may need to take~$k$ as large as~$O(M)^{J+d}$.
\par
We next use the Fourier averaging technique described in~\cite{MR948108};
see also~\cite{MR1104702,MR2259240}.
Let $\gp_v\mid\gD_1$ be a prime at which~$E$ has split multiplicative
reduction. The group of components of the special fiber of the N\'eron
model of~$E$ at~$v$ is a cyclic group of order
\[
  n_v = \ord_v\bigl(\gD(E/K)\bigr),
\]
and we let $0\le a_v(P)<n$ be the component that is hit by~$P$.  (In
practice, there is no prefered orientation to the cyclic group of
components, so~$a_v(P)$ is only defined up to~$\pm1$. This will not
affect our computations.)  The formula for the local height at a split
multiplicative place (due to Tate, see~\cite[VI.4.2]{MR1312368})
says that
\[
  \lhat(P;v) \ge \frac{1}{2}\Bern\left(\frac{a_v(P)}{n_v}\right)
             \log\Norm_{K/\QQ}\gp_v^{n_v}.
\]
In this formula,~$\Bern(t)$ is the periodic second Bernoulli
polynomial, equal to~$t^2-t+\frac16$ for $0\le t\le 1$ and extended
periodically modulo~$1$.  As in~\cite{MR948108}, we are going to take a
weighted sum of this formula over~$mP$ for $1\le m\le M$.
\par
The periodic Bernoulli polynomial has a Fourier expansion
\[
  \Bern(t) = \frac{1}{2\pi^2}
         \sum_{\substack{n\ge1\\n\ne0\\}} \frac{e^{2\pi i nt}}{n^2}
    = \frac{1}{\pi^2} \sum_{n=1}^\infty \frac{\cos(2\pi nt)}{n^2}.
\]
We also use the formula (Fej\'er kernel)
\[
  \sum_{m=1}^M \left(1-\frac{m}{M+1}\right) \cos(mt)
  = \frac{1}{2(M+1)}\biggl|\sum_{m=0}^M e^{imt}\biggr|^2 - \frac{1}{2}.
\]
Hence
\begin{align*}
  \sum_{m=1}^M &\left(1-\frac{m}{M+1}\right)\lhat(mP;v) \\*
  &\ge \sum_{m=1}^M \left(1-\frac{m}{M+1}\right) 
             \frac{1}{2}\Bern\left(\frac{ma_v(P)}{n_v}\right)
             \log\Norm_{K/\QQ}\gp_v^{n_v} \\
  &= \sum_{m=1}^M \left(1-\frac{m}{M+1}\right) 
      \frac{1}{2\pi^2} \sum_{n=1}^\infty \frac{\cos(2\pi nma_v(P)/n_v)}{n^2}\\
  &= \frac{1}{2\pi^2} \sum_{n=1}^\infty \frac{1}{n^2}
         \sum_{m=1}^M \left(1-\frac{m}{M+1}\right) 
             \cos\left(\frac{2\pi nma_v(P)}{n_v}\right) \\*
  &= \frac{1}{2\pi^2} \sum_{n=1}^\infty \frac{1}{n^2} \left(
     \frac{1}{2(M+1)}\biggl|\sum_{m=0}^M e^{2\pi inma_v(P)/n_v}\biggr|^2 
           - \frac{1}{2} \right).
\end{align*}
We split the sum over~$n$ into two pieces. If $n$ is a multiple of~$n_v$,
then the quantity between the absolute value signs is equal to~$M+1$,
and if~$n$ is not a multiple of~$n_v$, we simply use the fact that
the absolute value is non-negative. This yields the local estimate
\begin{align*}
  \sum_{m=1}^M &\left(1-\frac{m}{M+1}\right)\lhat(mP;v)\\
  &\ge  \left(\frac{1}{4\pi^2(M+1)} \sum_{n=1}^\infty \frac{(M+1)^2}{(nn_v)^2} 
     - \frac{1}{4\pi^2} \sum_{n=1}^\infty \frac{1}{n^2} 
         \right)\log\Norm_{K/\QQ}\gp_v^{n_v} \\
  &= \left(\frac{(M+1)}{24n_v^2} - \frac{1}{24}\right)
       \log\Norm_{K/\QQ}\gp_v^{n_v}.
\end{align*}
\par
We next sum the local heights over all primes dividing~$\gD_1$,
\begin{align*}
  \sum_{\gp_v\mid\gD_1}
     \sum_{m=1}^M \left(1-\frac{m}{M+1}\right)&\lhat(mP;v)\\*
  &\ge \frac{1}{24}
    \sum_{\gp_v\mid\gD_1} \left(\frac{(M+1)}{n_v} - n_v\right)
       \log\Norm_{K/\QQ}\gp_v .
\end{align*}
We set
\[
  M+1 = 2 {\dsum_{\gp_v\mid\gD_1} n_v\log\Norm_{K/\QQ}\gp_v}
          \bigg/
        {\dsum_{\gp_v\mid\gD_1} n_v^{-1}\log\Norm_{K/\QQ}\gp_v},
\]
which gives the height estimate
\begin{align*}
  \sum_{\gp_v\mid\gD_1}
     \sum_{m=1}^M \left(1-\frac{m}{M+1}\right)\lhat(mP;v)
  &\ge \frac{1}{24}
    \sum_{\gp_v\mid\gD_1}  n_v \log\Norm_{K/\QQ}\gp_v \\
  &=  \frac{1}{24}
    \sum_{\gp_v\mid\gD_1} \log\left|\Norm_{K/\QQ}\gD(E/K)\right|_v^{-1} .
\end{align*}
We also need to estimate the size of~$M$. This is done using
the elementary inequality 
\begin{equation}
  \label{eqn:jensen}
  \biggl(\sum_{i=1}^n a_ix_i\biggr)\biggl(\sum_{i=1}^n a_ix_i^{-1}\biggr)
  \ge \biggl(\sum_{i=1}^n a_i\biggr)^2,
\end{equation}
valid for all $a_i,x_i>0$. (This is a special case of Jensen's
inequality, applied to the function~$1/x$.)
Applying~\eqref{eqn:jensen} with $x_i=n_v$
and~$a_i=\log\Norm_{K/\QQ}\gp_v$ yields
\[
  M+1 \le 
   2 \left(\frac{\dsum_{\gp_v\mid\gD_1} 
         n_v\log\Norm_{K/\QQ}\gp_v}
      {\dsum_{\gp_v\mid\gD_1} 
         \log\Norm_{K/\QQ}\gp_v}\right)^2  
    =  2\s(\gD_1)^2 = 2\s_J\bigl(\gD(E/K)\bigr)^2.
\]
In particular, the chosen value of~$M$ is bounded solely in terms
of $\s_J\bigl(\gD(E/K)\bigr)$.
\par
We now combine the estimates for the local heights to obtain
\begin{align*}
  \sum_{m=1}^M& \left(1-\frac{m}{M+1}\right)\hhat(mkP)\\
  &\ge \sum_{m=1}^M \left(1-\frac{m}{M+1}\right) \biggl(
       \sum_{v\in M_K^\infty}+\sum_{\gp_v\mid\gD(E/K)} \biggr) \lhat(mkP;v) \\
  &= \biggl(\sum_{v\in M_K^\infty}+\sum_{\gp_v\mid\gD_1}+\sum_{\gp_v\mid\gD_2}
         \biggr)
     \sum_{m=1}^M \left(1-\frac{m}{M+1}\right)\lhat(mkP;v) \\
  &\ge \sum_{v\in M_K^\infty} \sum_{m=1}^M \left(1-\frac{m}{M+1}\right)
              \left(c_1\log\max\bigl\{ |j(E)|_v , 1 \bigr\}  -  c_2\right) \\*
  &\qquad{}+
    \frac{1}{24}\sum_{\gp_v\mid\gD_1}
          \log\bigl|\Norm_{K/\QQ}\gD(E/K)\bigr|_v^{-1} \\*
  &\qquad{}+ \sum_{\gp_v\mid\gD_2}\sum_{m=1}^M \left(1-\frac{m}{M+1}\right)
     c_3 \log\bigl|\Norm_{K/\QQ}\gD(E/K)\bigr|_v^{-1} \\*
  &\ge c_4h\bigl(j(E)\bigr) + c_5\log\Norm\gD(E/K) - c_6.
\end{align*}
In the last line we have used the fact that $\gD(E/K)j(E)$ is integral,
so
\[  
  \sum_{v\in M_K^\infty} \log\max\bigl\{ |j(E)|_v , 1 \bigr\}
  + \sum_{\gp_v\mid\gD_1\gD_2} \log\bigl|\Norm_{K/\QQ}\gD(E/K)\bigr|_v^{-1}
  \ge h\bigl(j(E)\bigr).
\]
On the other hand, 
\begin{align*}
  \sum_{m=1}^M \left(1-\frac{m}{M+1}\right)\hhat(mkP)
  &= \sum_{m=1}^M \left(1-\frac{m}{M+1}\right)m^2k^2\hhat(P)\\*
  &= \frac{k^2M(M+1)(M+2)}{12}\hhat(P).
\end{align*}
Adjusting the constants yet again yields
\[
  \hhat(P) \ge 
  \frac{c_7h\bigl(j(E)\bigr) + c_8\log\Norm_{K/\QQ}\gD(E/K) - c_9}{k^2M^3}
  \ge \frac{c_{10}h(E/K)-c_9}{k^2M^3}.
\]
Since~$M$ depends only on~$\s_J\bigl(\gD(E/K)\bigr)$ and since
$k\le(6M)^{J+d}$, this gives the desired lower bound for~$\hhat(P)$.
\end{proof}

\begin{remark}
As in~\cite{MR2259240}, a similar argument can be used to prove
that $\#E(K)_\tors$ is bounded by a constant that depends only
on~$[K:\QQ]$,~$J$, and~$\s_J\bigl(\Disc(E/K)\bigr)$.
\end{remark}

\section{Some elementary properties of the prime-depleted Szpiro ratio}
\label{section:properties}

We start with an elementary inequality.

\begin{lemma}
\label{lemma:ineq}
Let $n\ge2$, and let $\a_1,\ldots,\a_n$ and $x_1,\ldots,x_n$ be
positive real numbers, labeled so that $\a_n=\max\a_i$. Then
\[
  \frac{\a_1x_1+\cdots+\a_nx_n}{x_1+\cdots+x_n}
  \ge
  \frac{\a_1x_1+\cdots+\a_{n-1}x_{n-1}}{x_1+\cdots+x_{n-1}},
\]
with strict inequality unless $\a_1=\cdots=\a_n$.
\end{lemma}
\begin{proof}
Let $A=\sum_{i=1}^n\a_ix_i$ and $X=\sum_{i=1}^nx_i$. Then
\begin{align}
  \label{eqn:AXxn}
  A(X-x_n)-(A-\a_nx_n)X 
  &= (\a_n X- A)x_n \notag \\*
  &= \biggl(\sum_{i=1}^n (\a_n-\a_i)x_i\biggr)x_n \ge 0.
\end{align}
Hence
\begin{equation}
  \label{eqn:AXAanxn}
  \frac{A}{X} \ge \frac{A-\a_nx_n}{X-x_n},
\end{equation}
and since the~$x_i$ are assumed to be positive,
inequalities~\eqref{eqn:AXxn} and~\eqref{eqn:AXAanxn} are strict
unless the~$\a_i$ are all equal.
\end{proof}

We apply the lemma to prove some basic properties of the $J$-depleted
Szpiro ratio.

\begin{proposition}
\label{proposition:sproperties}
Let $J\ge1$.
\begin{parts}
\Part{(a)}
For all integral ideals~$\gD$,
\[
  \s_{J-1}(\gD) \ge \s_{J}(\gD).
\]
Further, the inequality is strict unless~$\gD$ has the
form $\gD=\gI^e$ for a squarefree ideal~$\gI$.
\Part{(b)}
Assume that $\nu(\gD)\ge J$. Then there exists an ideal $\gd\mid\gD$
satisfying
\[
  \nu(\gd)=J\qquad\text{and}\qquad \s_J(\gD)=\s(\gD/\gd).
\]
\Part{(c)}
Let~$\gp$ be a prime ideal and~$\gD$ an ideal
with~$\gp\notdivide\gD$. Then
\[
  \s_J(\gD)
  \ge \s_J(\gp^e\gD) 
  \ge \frac{\s_J(\gD)}{\log\Norm_{K/\QQ}\gp}.
\]
\Part{(d)}
Let~$\gp$ be a prime ideal and let~$\gD$ an arbitrary ideal 
\textup(so~$\gp$ is allowed to divide~$\gD$\textup). Then
\[
  (\log\Norm_{K/\QQ}\gp) \s_J(\gD)
  \ge \s_J(\gp^e\gD) 
  \ge \frac{\s_J(\gD)}{\log\Norm_{K/\QQ}\gp}.
\]
\end{parts}
\end{proposition}
\begin{proof}
(a)\enspace
Write $\gD=\prod \gp_i^{e_i}$. To ease notation, we let
\[
  q_i = \log\Norm_{K/\QQ}\gp_i.
\]
If $\nu(\gD)\le J-1$, then~$\s_{J-1}(\gD)=\s_J(\gD)=1$, so there is
nothing to prove.  Assume now that $\nu(\gD)\ge J$.
Let $I\subset\bigl\{1,2,\ldots,\nu(\gD)\bigr\}$ be a set
of indices with $\#I\ge\nu(\gD)-(J-1)$ satisfying
\[
  \s_{J-1}(\gD) = \sum_{i\in I} e_iq_i \bigg/ \sum_{i\in I} q_i.
\]
Let $k\in I$ be an index satisfying $e_k=\max\{ e_i:i\in I\}$.
Then Lemma~\ref{lemma:ineq} with $\a_i=e_i$ and $x_i=q_i$ yields
\[
  \s_{J-1}(\gD) 
  = \frac{\dsum_{i\in I} e_iq_i}{\dsum_{i\in I} q_i}
  \ge \frac{\dsum_{i\in I,\,i\ne k} e_iq_i}{\dsum_{i\in I,\,i\ne k} q_i}
  \ge \s_J(\gD).
\]
Further, Lemma~\ref{lemma:ineq} says that the  inequality is strict
unless all of the~$e_i$ are equal, in which case~$\gD$ is a power of
a squarefree ideal.
\ProofPart{(b)}
If~$\gD=\gI^e$ is a power of a squarefree ideal, then
$\s_J(\gD)=\s(\gD/\gc^e)$ for every ideal $\gc\mid\gI$
satisfying~$\nu(\gc)=J$, so the assertion to be proved is
clear.  We may thus assume that~$\gD$ is not a power of a squarefree
ideal.
\par
Suppose in this case that $\s_J(\gD)=\s(\gD/\gd)$ for some
$\gd\mid\gD$ with $\nu(\gd)\le J-1$. Then
\[
  \s_{J-1}(\gD) \le \s(\gD/\gd) = \s_J(\gD),
\]
contradicting the fact proven in~(a) that $\s_{J-1}(\gD)>\s_J(\gD)$ (strict
inequality).
\ProofPart{(c)}
We always have
\[
  \s_J(\gp^e\gD) \le \s_{J-1}(\gD),
\]
since in computing~$\s_J(\gp^e\gD)$, we can always remove~$\gp$
and~$J-1$ other primes from~$\gD$. If this inequality is an equality,
we're done.  Otherwise the value of~$\s_J(\gp^e\gD)$ is obtained by
removing~$J$ primes from~$\gD$.  Continuing with the notation from~(a)
and letting $q=\log\Norm_{K/\QQ}\gp$, this means that there is an
index set~$I$ with $\#I\ge\nu(\gD)-J$ such that
\[
  \s_{J}(\gD) 
  = \frac{eq + \dsum_{i\in I} e_iq_i }{ q + \dsum_{i\in I} q_i}
  \ge \frac{q + \dsum_{i\in I} e_iq_i }{ q + \dsum_{i\in I} q_i}
  = \frac{q+X}{q+Y},
\]
where to ease notation, we write~$X$ and~$Y$ for the indicated sums.
\par
If $Y=0$, then also~$X=0$ and $\nu(\gD)\le J$, so $\s_J(\gp^e\gD)$ 
equals either~$e$ or~$1$. In either case, it is greater than~$\s_J(\gD)=1$.
So we may assume that~$Y>0$, which implies that $Y\ge\log2$.
\par
We observe that
\[
  \frac{X}{Y} = \frac{\dsum_{i\in I} e_iq_i }{  \dsum_{i\in I} q_i}
  \ge \s_J(\gD).
\]
Hence
\[
  \s_J(\gD) = \frac{X}{Y}\cdot\frac{1+q/X}{1+q/Y}
  \ge \frac{\s_J(\gD)}{1+q/Y} \ge \frac{\s_J(\gD)}{3q}.
\]
(The final inequality is true since $q\ge\log2$ and $Y\ge\log2$.)
This proves that~$\s_J(\gD)$ is greater than the smaller of~$\s_{J-1}(\gD)$
and~$\s_J(\gD)/3q$. But from~(a) we have~$\s_{J-1}(\gD)\ge\s_J(\gD)$, so
the latter is the minimum.
\ProofPart{(d)}
Let $\gD=\gp^i\gD'$ with $\gp\notdivide\gD'$. Then 
writing $q=\log\Norm_{K/\QQ}\gp$ as usual and
applying~(c)
several times, we have
\[
  \s_J(\gp^e\gD) = \s_J(\gp^{e+i}\gD')
  \le \s_J(\gD') 
  \le q\s_J(\gp^i\gD')
  = q\s_J(\gD).
\]
Similarly
\[
  \s_J(\gp^e\gD) = \s_J(\gp^{e+i}\gD')
  \ge \frac{\s_J(\gD')}{q}
  \ge \frac{\s_J(\gp^i\gD')}{q}  
  = \frac{\s_J(\gD)}{q}.
\]
\end{proof}

\section{The Prime-Depleted Szpiro and $ABC$ Conjectures}
\label{section:abc}

In this section we describe a prime-depleted variant of the
$ABC$-conjecture and show that it is a consequence of the
prime-depleted Szpiro conjecture.  For ease of notation, we restrict
attention to $K=\QQ$ and leave the generalization to arbitrary fields
to the reader.

\begin{conjecture}[Prime-Depleted $ABC$-conjecture]
There exist an integer~$J\ge0$ and a constant~$C_5$ such that 
if~$A,B,C\in\ZZ$ are integers satisfying
\[
  A+B+C=0\qquad\text{and}\qquad\gcd(A,B,C)=1,
\]
then
\[
  \s_J(ABC) \le C_5.
\]
\end{conjecture}

The classical $ABC$-conjecture (with non-optimal exponent) says
that $\s(ABC)$ is bounded, which is stronger than the prime-depleted
version, since $\s_J(ABC)$ is less than or equal to~$\s(ABC)$.

\begin{proposition}
If the prime-depleted Szpiro conjecture is true, then the
prime-depleted $ABC$-conjecture is true.
\end{proposition}
\begin{proof}
We suppose  that the prime-depleted Szpiro conjecture is true, say
with~$J$ primes deleted. Let~$A,B,C\in\ZZ$ be as in the statement
of the depleted $ABC$-conjecture. We consider the Frey curve
\[
  E:y^2 = x(x+A)(x-B).
\]
An easy calculation~\cite[VIII.11.3]{MR1329092} shows that the minimal
discriminant of~$E$ is either $2^4(ABC)^2$ or $2^{-8}(ABC)^2$, so in
any case we can write~$\Disc(E/\QQ)=2^e(ABC)^2$ for some
exponent~$e\in\ZZ$. Then using Proposition~\ref{proposition:sproperties}
we find that
\[
  \s_J\bigl(\Disc(E/\QQ)\bigr) 
  = \s_J\bigl(2^e(ABC)^2\bigr)
  \ge \frac{\s_J\bigl((ABC)^2\bigr)}{\log2}
  = \frac{2\s_J(ABC)}{\log 2}.
\]
Hence the boundedness of~$\s_J\bigl(\Disc(E/\QQ)\bigr)$ implies the
boundedness of $\s_J(ABC)$.
\end{proof}

\begin{remark}
The Szpiro and $ABC$-conjectures have many important consequences,
including asymptotic Fermat (trivial), a strengthened version of
Roth's theorem~\cite{bombieri,MR1924103}, the infinitude of
non-Wieferich primes \cite{MR961918}, non-existence of Siegel
zeros~\cite{MR1738058}, Faltings' theorem (Mordell
conjecture)~\cite{MR1141316,MR1924103},\dots.  (For a longer
list, see~\cite{nitaj}.) It is thus of interest to ask which, if any,
of these results follows from the prime-depleted Szpiro conjecture. As
far as the author has been able to determine, the answer is none of
them, which would seem to indicate that the prime-depleted Szpiro
conjecture is qualitatively weaker than the original Szpiro
conjecture.
\end{remark}





\end{document}